\documentclass{bams-arx}

\usepackage{latexsym, amsfonts, amsmath, amsthm, amssymb,verbatim,amscd}	
\newcommand\thedate{\textsl{18/10/2018}}

\citesort

\newcommand{\vecx}{{\boldsymbol{x}}}

\newcommand{\st}{\,:\,}
\newcommand{\norm}[1]{\left\Vert #1 \right\Vert}

\newcommand{\AN}{\mathcal{AN}}
\newcommand{\esspec}{\sigma_{ess}}

\theoremstyle{cupthm}
\newtheorem{thm}{Theorem}   %[section]

\theoremstyle{cupdefn}

\theoremstyle{cuprem}

\numberwithin{equation}{section}

\newcommand{\journalname}[1]{{\textit{#1}}}

\begin{document}
%%%%%%%%%%%%%%%%%%%%%%%%%%%%%%%%%%%%%%%%%%%%%%%%%%%%%%%%%%%%%%%%%%%%%%%
%
%   Title
%
%%%%%%%%%%%%%%%%%%%%%%%%%%%%%%%%%%%%%%%%%%%%%%%%%%%%%%%%%%%%%%%%%%%%%%%
\runningtitle{Positive $\mathcal{AN}$ operators}
\title{A note on positive $\mathcal{AN}$ operators}

\author[1]{Ian Doust}
\address[1]{School of Mathematics and Statistics,
University of New South Wales,\
UNSW Sydney 2052 Australia \email{i.doust@unsw.edu.au}}%
\authorheadline{I. Doust}

%%%%%%%%%%%%%%%%%%%%%%%%%%%%%%%%%%%%%%%%%%%%%%%%%%%%%%%%%%%%%%%%%%%%%%%
%
%   The abstract
%
%%%%%%%%%%%%%%%%%%%%%%%%%%%%%%%%%%%%%%%%%%%%%%%%%%%%%%%%%%%%%%%%%%%%%%%
\begin{abstract}
We show that positive absolutely norm attaining operators can be characterized by a simple property of their spectra. This result clarifies and simplifies a result of Ramesh. As an application we characterize weighted shift operators which are absolutely norm attaining.
\end{abstract}

%\email{i.doust@unsw.edu.au}%
%\date{}
%\date{}
\maketitle

%%%%%%%%%%%%%%%%%%%%%%%%%%%%%%%%%%%%%%%%%%%%%%%%%%%%%%%%%%%%%%%%%%%%%%%
%
%   Introduction
%
%%%%%%%%%%%%%%%%%%%%%%%%%%%%%%%%%%%%%%%%%%%%%%%%%%%%%%%%%%%%%%%%%%%%%%%

%\section{Introduction}\label{intro}

A bounded linear operator $T$ on a complex Hilbert space $H$ is said to be absolutely norm attaining if given any nonzero subspace $M$ of $H$, there exists $\vecx_0$ in the unit ball $M_1$ of $M$ such that $\norm{T\vecx_0} = \sup\{ \norm{T\vecx} \st \vecx \in M_1\}$. The set of all absolutely norm attaining operators, which we shall denote by $\AN$, was introduced by Carvajal and Neves \cite{CN}. As was shown by Pandey and Paulsen, there are severe restrictions on the structure of such operators.

\begin{thm}\label{thm1} \textrm{\cite[Theorem~5.1]{PP}} 
Suppose that $T$ is a positive operator on $H$. Then $T \in \AN$ if and only if $T = \alpha I + K + F$ where $\alpha\ge 0$, $K$ is a positive compact operator and $F$ is a self-adjoint finite-rank operator.
\end{thm}

In \cite{R}, Ramesh proposes a different characterization of positive $\AN$ operators. Unfortunately Theorem~2.4 of \cite{R} is misstated, and is perhaps more complicated than it needs to be. The main point of Theorem~\ref{thm2} below is that one only needs to check two elementary properties of the spectrum of an operator to ensure that it is of the form described in Theorem~\ref{thm1}.  

Let $\mathcal{C}(H) = B(H)/K(H)$ denote the Calkin algebra with quotient map $\pi$, and recall that 
the essential spectrum of an operator $T$, denoted $\esspec(T)$, is the spectrum of $\pi(T)$ in $C(H)$.
%We shall denote the essential spectrum of an operator $T$ by $\esspec(T)$.

\begin{thm}\label{thm2}
Suppose that $T$ is a positive operator on an infinite dimensional Hilbert space $H$. Then $T \in \AN$ if and only if $\esspec(T)$ contains a single point $\alpha$ and $\sigma(T)$ contains only finitely many elements less than $\alpha$.
\end{thm}

\begin{proof}
The forward direction is a direct consequence of standard results about the invariance of essential spectrum under compact perturbations (see for example, \cite[Section~XI.4]{C}), and the proof of Theorem~\ref{thm1} (or Theorem~3.25 of \cite{CN}).

Conversely, suppose that $T$ is a positive operator and that $\esspec(T) = \{\alpha\}$. This implies that $\pi(T-\alpha I)$ is a quasinilpotent self-adjoint element of $C(H)$ and hence is zero. That is $T = \alpha I + K$ where $K$ is a compact self-adjoint operator. The spectral theorem for such operators says that $K = \sum_{n \in N} \lambda_n P_n$ where $N$ is a countable set and $P_n$ is the orthogonal  finite rank projection onto the eigenspace for the eigenvalue $\lambda_n$.

% Conversely, suppose that $T$ is a positive operator and that $\esspec(T) = \{\alpha\}$. This implies that $\esspec(T-\alpha I) = \{0\}$ and hence that $T - \alpha I$ is a Riesz operator. But $T - \alpha I$ is also self-adjoint and these two properties ensure that $T - \alpha I$ is compact. Thus $T = \alpha I + K$ where $K$ is a compact self-adjoint operator. The spectral theorem for such operators says that $K = \sum_{n \in N} \lambda_n P_n$ where $N$ is a countable set and $P_n$ is the orthogonal  finite rank projection onto the eigenspace for the eigenvalue $\lambda_n$. 

Let $N^{-} = \{n \st \lambda_n < 0\}$ and $N^+ = \{n \st \lambda_n \ge 0\}$. Since $\sigma(T) = \{\alpha\} \cup \{\alpha + \lambda_n \st n \in N\}$, if $\sigma(T)$ contains only finitely many elements less than $\alpha$ then $N^{-}$ is a finite set and hence $F = \sum_{n \in N^{-}} \lambda_n P_n$ is self-adjoint and finite rank (with of course the convention that an empty sum is zero). The operator $K^+ = \sum_{n \in N^{+}} \lambda_n P_n$ is compact and positive. Since $T = \alpha I + K^+ +F$ we can apply Theorem~\ref{thm1} to deduce that $T \in \AN$.
\end{proof}

Pandey and Paulsen noted \cite[Lemma~6.2]{PP} that $T \in \AN$ if and only if $|T| = (T^*T)^{1/2} \in \AN$ so one may write a corresponding characterization of general $\AN$ operators in terms of the spectral properties of $|T|$. 

As a simple application, it follows that a (bounded) weighted shift operator on $\ell^2$
  \[ T(x_1,x_2,x_3,\dots) = (0,w_1x_1,w_2 x_2, \dots)\]
is absolutely norm attaining if and only if either
\begin{enumerate}
 \item[(i)] the set $\bigl\{\,|w_n|\,\bigr\}_{n=1}^\infty$ has a unique limit point $\alpha$;
 \item[(ii)] for all $\beta \ne \alpha$, $|w_n| = \beta$ for only finitely many values of $n$; and
 \item[(iii)] $|w_n| < \alpha$ for only finitely many values of $n$,
\end{enumerate}
or 
\begin{enumerate}
 \item[(i')] $\sigma(|T|) = \bigl\{\,|w_n|\,\bigr\}_{n=1}^\infty$ is a finite set; and
 \item[(ii')] there is only one value of  $\beta \in \sigma(|T|)$ such that $|w_n| = \beta$ for infinitely many values of $n$.
\end{enumerate}
Some related results can be found in \cite{L}

%%%%%%%%%%%%%%%%%%%%%%%%%%%%%%%%%%%%%%%%%%%%%%%%%%%%%%%%%%%%%%%%%%%%%%%
%
%   Bibliography
%
%%%%%%%%%%%%%%%%%%%%%%%%%%%%%%%%%%%%%%%%%%%%%%%%%%%%%%%%%%%%%%%%%%%%%%%

\end{document}